\documentclass[12pt]{amsart}
\newtheorem{theorem}{Theorem}[section]
\newtheorem{lemma}[theorem]{Lemma}

\newtheorem{proposition}[theorem]{Proposition}
\newtheorem{corollary}[theorem]{Corollary}

\theoremstyle{definition}

\newtheorem{remark}[theorem]{Remark}

\newtheorem{example}[theorem]{Example}

\usepackage{epsfig,color}

\makeatother

\theoremstyle{definition}

\def\fnum{equation}

\newtheorem{Pro}[\fnum]{Proposition}

\numberwithin{equation}{section}

\usepackage[hidelinks]{hyperref}

\begin{document}
\title
[Disks area-minimizing in mean convex Riemannian $n$-manifolds]
{Disks area-minimizing in mean convex Riemannian $n$-manifolds}

\author{Ezequiel Barbosa}
\address{Universidade Federal de Minas Gerais (UFMG), Caixa Postal 702, 30123-970, Belo Horizonte, MG, Brazil}
\email{ezequiel@mat.ufmg.br}
\author[Franciele Conrado]{Franciele Conrado$^{2}$} \address{$^{2}$Instituto de Ci\^{e}ncias Exatas-Universidade Federal de Minas Gerais\\ 30161-970-Belo Horizonte-MG-BR} \email{franconradomat@gmail.com} \thanks{$^{2}$ Partially supported by CNPq} 

\begin{abstract}
We prove the validity of an inequality involving a mean of the area and the length of the boundary of immersed disks whose boundaries are homotopically non-trivial curves in an oriented compact manifold which possesses convex mean curvature boundary, positive escalar curvature and admits a map to $\mathbb{D}^2\times T^{n}$ with nonzero degree, where $\mathbb{D}^2$ is a disk and $T^n$ is an $n$-dimensional torus. We also prove a rigidity result for the equality case when the boundary is totally geodesic. This can be viewed as a partial generalization of a result due to Lucas Ambr\'ozio in \cite{AMB} to higher dimensions.
\end{abstract}

\maketitle

\section{Introduction}\label{intro}

An important question in modern differential geometry is about the connection between the curvatures and topology of a manifold. 
A very significant and historic result on this is the famous Gauss-Bonnet theorem. As a consequence of that theorem, we note that the topological invariant, named Euler Characteristic, gives a topological obstruction to the existence of a certain type of Riemannian metrics on surfaces. In higher dimensions, the relationship between curvatures and the topology of a manifold is much more complicated. However, Schoen and Yau, in their celebrated joint work, discovered interesting relations between the scalar curvature of a three-dimensional manifold and the topology of stable minimal surfaces inside it, which emerge when one uses the second variation formula for the area, the Gauss equation and the Gauss-Bonnet theorem.

In a very recent paper Bray, Brendle and Neves \cite{BBN} proved an elegant rigidity result concerning to an area-minimising 2-sphere embedded in a closed  3-dimensional manifold $(M^3,g)$ with positive scalar curvature and $\pi_2(M)\neq0$. In that work, they showed the following result. Denote by $\mathcal{F}$ the set of all smooth maps $f:\mathbb{S}^2\rightarrow M$ which represent a nontrivial element in $\pi_2(M)$. Define
\[
\mathcal{A}(M,g)=\inf \{ Area(\mathbb{S}^2, f^*g)\,:\,\,f\in \mathcal{F} \}\,.
\]

If $R_g\geq2$, the following inequality holds:
\[
\mathcal{A}(M,g)\leq 4\pi\,,
\]
where $R_g$ denote the scalar curvature of $(M,g)$. Moreover, if the equality holds then the universal cover of $(M,g)$ is isometric to the standard cylinder $\mathbb{S}^2\times \mathbb{R}$ up to scaling. For more results concerning to rigidity of 3-dimensional closed manifolds coming from area-minimising surfaces, see \cite{BBNE}, \cite{CG}, \cite{MM}, \cite{RM}, \cite{Nunes}.  In \cite{JZ}, J. Zhou showed a version of Bray, Brendle and Neves \cite{BBN} result for high co-dimension:   for $n + 2 \leq 7$, let $(M^{n+2}, g)$ be an oriented closed Riemannian manifold with $R_g \geq 2$, which admits a non-zero degree map $F:M\rightarrow \mathbb{S}^2\times T^n$. Then $\mathcal{A}(M,g)\leq4\pi$. Furthermore, the equality implies that the universal covering  of $(M^{n+2},g)$ is $\mathbb{S}^2\times \mathbb{R}^n$.

In the same direction as the results mentioned above for the closed manifolds, let $M$ be a Riemannian manifold with nonempty boundary $\partial M$.  A free boundary minimal surface in $M$ is a minimal surface in $M$ with boundary contained in the boundary $\partial M$ and meeting it orthogonally. Such surfaces arise variationally as critical points of the area among surfaces in $M$ whose boundaries lie on $\partial M$ but are free to vary on $\partial M$. The simplest examples, considering $M$ as the unit ball with center at the origin in the Euclidean space, are an equatorial plane disk and the critical catenoid, the unique piece
of a suitably scaled catenoid in the unit ball.  A. Fraser and R. Schoen \cite{FS3} established a connection between free boundary minimal surfaces and the Steklov eigenvalue problem, and proved existence of an embedded free boundary minimal surface of genus zero with any number of boundary components. Since then, many works was developed to study free-boundary minimal surfaces. For more results concerning free boundary minimal surfaces, see the following references and the references therein: \cite{AMB}, \cite{FPeng}, \cite{AF1}, \cite{AF2}, \cite{AF3}, \cite{FS1}, \cite{FS2}, \cite{FS3}, \cite{FS4}, \cite{FPZ}, \cite{MLi}.   

Consider now a  Riemannian $n$-manifold with non-empty boundary $(M,\partial M,g)$. Let $\mathcal{F}_M$ be the set of all immersed disks in $M$ whose boundaries are curves in $\partial M$ that are homotopically non-trivial in $\partial M$. If $\mathcal{F}_M\not=\emptyset$, we define

\[\mathcal{A}(M,g)=\inf_{\Sigma\in \mathcal{F}_M} |\Sigma|_g \ \ \text{and} \ \ \mathcal{L}(M,g)=\inf_{\Sigma\in \mathcal{F}_M} |\partial \Sigma|_g\]

In the paper \cite{AMB}, L. C. Ambr\'ozio proved the following result.

\begin{theorem}
Let $(M,g)$ be a compact Riemannian $3$-manifold with mean convex boundary. Assume that $\mathcal{F}_M\not=\emptyset$. Then 
\begin{equation}\label{equ0}
\frac{1}{2}\inf R_g^M \mathcal{A}(M,g)+\inf H_g^{\partial M}\mathcal{L}(M,g)\leq 2\pi.
\end{equation}
Moreover, if equality holds, then the universal covering of $(M,g)$ is isometric to $(\mathbb{R}\times \Sigma_0, dt^2+g_0)$, where  $(\Sigma_0,g_0)$ is a disk with constant Gaussian curvature $\frac{1}{2}\inf R_g$ and $\partial\Sigma_0$ has constant geodesic curvature $\inf H_g^{\partial M}$ in $(\Sigma_0,g_0)$.
\end{theorem}

A question that arises here is the following: {\it Is it possible to obtain similar result for high co-dimension?}  Unfortunately, a general result cannot be true as we can see with the following example. Consider $(M,g)=(\mathbb{S}^2_+(r)\times\mathbb{S}^m(R), h_0+g_0),$ where  $(\mathbb{S}^2_+(r),h_0)$ is the half $2$-sphere of radius $r$ with the standard metric, and $(\mathbb{S}^m(R),g_0)$  is the $m$-sphere of radius $R$ with the standard metric, $m\geq 2$. This case, we have that
\[
\frac{1}{2}\inf R_g^M \mathcal{A}(M,g)+\inf H_g^{\partial M}\mathcal{L}(M,g) >2\pi.\]

On the other hand, consider $(M,g)=(\mathbb{S}^2_+(r)\times T^m, g_0+\delta),$ where $(T^m,\delta)$  is the flat $m$-torus, $m\geq 2$. Note that the equality holds in (\ref{equ0}). However, we can see that in this case the universal covering of $(M,g)$ is isometric to $(\mathbb{S}^2_+(r)\times \mathbb{R}^m, g_0+\delta_0)$, where $\delta_0$ is a standard metric in $\mathbb{R}^m$.

In the first example above, note that there is no map $F:(M,\partial M)\rightarrow (\mathbb{D}^2\times T^n,\partial\mathbb{D}^2\times T^n)$ with non-zero degree. However, this is a condition that we need in order to obtain a similar result as in \cite{AMB}. Our main result of this work is the following.

\begin{theorem}
Let $(M,\partial M,g)$ be a Riemannian $(n+2)$-manifold, $3\leq n+2\leq 7$, with positive scalar curvature and mean convex boundary. Assume that there is a map $F:(M,\partial M)\rightarrow (\mathbb{D}^2\times T^n,\partial\mathbb{D}^2\times T^n)$ with non-zero degree. Then, 
\begin{equation}\label{EQ1}
\frac{1}{2}\inf R_g^M \mathcal{A}(M,g)+\inf H_g^{\partial M}\mathcal{L}(M,g)\leq 2\pi.
\end{equation}
Moreover, if the boundary $\partial M$ is totally geodesic and the equality holds in $(\ref{EQ1})$, then the universal covering of $(M,g)$ is isometric to $(\mathbb{R}^n\times \Sigma_0, \delta+g_0)$, where $\delta$ is the standard metric in $\mathbb{R}^n$ and $(\Sigma_0,g_0)$ is a disk with constant Gaussian curvature $\frac{1}{2}\inf R^M_g$ and $\partial\Sigma_0$ has null geodesic curvature in $(\Sigma_0,g_0)$.
\end{theorem}

This work is organised as follows. In Section 2, we present some
auxiliaries results to be used in the proof of the main results. In Section
3, we present the proof of the inequality in our main theorem 1.2. Finally, in Section 4, we present the proof of the rigidity part for the case where the equality is achieved and the manifold has totally geodesic boundary.

\subsection*{Acknowledgments} 
The first author was partially supported by CNPq-Brazil (Grant 312598/2018-1). The second author was partially supported by CAPES-Brazil (Grant 88882.184181/2018-01) and CNPq-Brazil (Grant 141904/2018-6).

\section{Free boundary minimal $k$-slicings}

All the manifolds considered here are compact and orientable. 

\subsection{Definition and Examples}
Let $(M,\partial M,g)$ be a Riemannian $n$-manifold. Assume there is a properly embedded free-boundary smooth hypersurface $\Sigma_{n-1}\subset M$ which minimizes volume in $(M,g)$. Choose $u_{n-1}>0$ a first eigenfunction for the second variation $S_{n-1}$ of the volume of $\Sigma_{n-1}$ in $(M,g)$. Define $\rho_{n-1}=u_{n-1}$ and the weighted volume functional $V_{\rho_{n-1}}$ for hypersurfaces of  $\Sigma_{n-1}$, 

\[V_{\rho_{n-1}}(\Sigma)=\int_{\Sigma} \rho_{n-1}dv_{\Sigma},\]
\noindent where $dv_{\Sigma}$ is the volume form in $(\Sigma,g)$. Assume there is a properly embedded free-boundary smooth hypersurface $\Sigma_{n-2}\subset\Sigma_{n-1}$ which minimizes the weighted volume functional $V_{\rho_{n-1}}$. Choose a first eigenfunction $u_{n-2}>0$ for the second variation $S_{n-2}$ of the weighted volume functional $V_{\rho_{n-1}}$ in $\Sigma_{n-2}$. Define $\rho_{n-2}=\rho_{n-1}u_{n-2}$. Assume that we can keep doing this, inductively. Hence, we obtain a family of free-boundary minimal smooth submanifolds
\[
\Sigma_k\subset \Sigma_{k+1}\subset\cdots\subset\Sigma_{n-1}\subset(\Sigma_n,g):=(M,g),
\]
which was constructed by choosing, for each $j\in\{k,\cdots,n-1\}$, a properly embedded free-boundary smooth hypersurface  $\Sigma_j\subset\Sigma_{j+1}$ which minimizes the weighted volume functional $V_{\rho_{j+1}}$, where $\rho_{j+1}:=\rho_{j+2}u_{j+1}=u_{j+1}u_{j+2}\cdots u_{n-1}$. We call such family of free-boundary minimal hypersurfaces a {\it free-boundary minimal $k$-slicing} in $(M,g)$.

\begin{example} Let $(N,\partial N,g)$ be a  Riemannian $k$-manifold. Consider the following Riemannian $n$-manifold $(N\times T^{n-k},g+\delta)$, where $\delta$ is the flat metric on the torus $T^{n-k}$. The family of smooth hypersurfaces   
\[
N\subset N\times S^1\subset N\times T^2\subset \cdots \subset N\times T^{n-k-1}\subset (N\times T^{n-k},g+\delta),
\]
where $\rho_j\equiv u_j\equiv 1$, for every $j=k,\cdots ,n-1$, is a free-boundary minimal $k$-slicing in $(N\times T^{n-k},g+\delta)$.
\end{example}

\subsection{Geometric formulas for free-boundary minimal $k$-slincing}

Let $(M,\partial M,g)$ be a Riemannian $n$-manifold. Consider a free-boundary $k$-slicing in $M$:
\[ 
\Sigma_k\subset \cdots \subset \Sigma_{n-1}\subset (\Sigma_n,g):=(M,g).
\]

{\bf Notation:}

\begin{itemize}

\item $Ric_j$:= Ricci curvature of $(\Sigma_j,g)$.
\item $R_j$:= Scalar curvature of $(\Sigma_j,g)$.
\item $\nu_j$:= Unit vector field of $\Sigma_j$ in $(\Sigma_{j+1},g)$.
\item $B_j$:= Second fundamental form of $\Sigma_j$ in $(\Sigma_{j+1},g)$.
\item $H_j$:= Mean curvature of $\Sigma_j$ in $(\Sigma_{j+1},g)$
\item $\eta_j$:= Outward unit vector field on the boundary $\partial\Sigma_j$ in $(\Sigma_j,g)$.
\item $B^{\partial \Sigma_j}$:= Second fundamental form of $\partial\Sigma_j$ in $(\Sigma_j,g)$ with respect to $\eta_j$.
\item $H^{\partial \Sigma_j}$:= Mean curvature of $\partial\Sigma_j$ in $(\Sigma_j,g)$ with respect to $\eta_j$.
\end{itemize}

\begin{remark}\label{obs1} Since $\Sigma_j$ is a free-boundary hypersurface in $(\Sigma_{j+1},g)$, for every $ j=k,\cdots,n-1$, we have that
\begin{enumerate}
\item $\eta_j=\eta_p$ in $\partial\Sigma_j$, for every $p\geq j$.
\item $H^{\partial \Sigma_j}=H^{\partial \Sigma_{j+1}}-B^{\partial \Sigma_{j+1}}(\nu_j,\nu_j)=H^{\partial M}-\displaystyle\sum_{p=j}^{n-1}B^{\partial \Sigma_{p+1}}(\nu_p,\nu_p).$
\end{enumerate}
\end{remark}

For each $j\in\{k,\cdots,n-1\}$, define on $\Sigma_j\times T^{n-j}$ the Riemannian metric 
\[ 
\hat{g}_j=g+\sum_{p=j}^{n-1}u_p^2dt_p^2.
\]

We define
\[\hat{\Sigma}_j=\Sigma_j\times T^{n-j} \ \ \text{and} \ \ \tilde{\Sigma}_j=\Sigma_j\times T^{n-j-1}.\]

Note that, since $\Sigma_j$ is free boundary hypersurafce in $(\Sigma_{j+1},g)$ we have that $\tilde{\Sigma}_j$ is free boundary hypersurface in $(\hat{\Sigma}_{j+1},\hat{g}_{j+1})$. With the next lemmas and propositions, we will prove that $\Sigma_j \times T^{n-j-1}$ in a free-boundary stable minimal hypersurface in $(\hat{\Sigma}_{j+1},\hat{g}_{j+1})$.


\begin{lemma}\label{l2} For every $j=k,\cdots,n-1$, the second fundamental form $\tilde{B}_j$ of $\tilde{\Sigma}_j$ in $(\hat{\Sigma}_{j+1},\hat{g}_{j+1})$ is given by
\[
\tilde{B}_j=B_j-\sum_{p=j+1}^{n-1}u_p\nu_j(u_p)dt_p^2.\]
In particular,
\[
|\tilde{B}_j|^2=|B_j|^2+\sum_{p=j+1}^{n-1}(\nu_j(\log u_p))^2.
\]
\end{lemma}
\begin{lemma}\label{l21} For every $j=k,\cdots,n-1$, the second fundamental form $\hat{B}_{j+1}$ of $\partial\hat{\Sigma}_{j+1}$ in $(\hat{\Sigma}_{j+1},\hat{g}_{j+1})$ with respect to $\eta_j$ satisfies
\[
\hat{B}_{j+1}(\nu_{j},\nu_{j})=B^{\partial\Sigma_{j+1}}(\nu_{j},\nu_j).
\]
\end{lemma}
\begin{lemma}\label{l22} For every $j=k,\cdots,n-1$, the Ricci Tensor $Ric_{\hat{g}_{j+1}}$ of $(\hat{\Sigma}_{j+1},\hat{g}_{j+1})$ satisfies
\[Ric_{\hat{g}_{j+1}}(\nu_j,\nu_j)=Ric_{j+1}(\nu_j,\nu_j)- \sum_{p=j+1}^{n-1} \frac{1}{u_p} \left(\nabla_{j+1}^2u_p\right)(\nu_j,\nu_j)\]
\noindent where $\nabla_{j+1}^2$ is the hessian in $(\Sigma_{j+1},g)$.
\end{lemma}
\begin{Pro} For every $j=k,\cdots,n-1$, $\tilde{\Sigma}_j$ is a free boundary minimal hypersurfaces in $(\hat{\Sigma}_{j+1},\hat{g}_{j+1})$.
\end{Pro}
\begin{proof}
Denote by $\tilde{H}_j$ the mean curvature of $\tilde{\Sigma}_j$ in $(\hat{\Sigma}_{j+1},\hat{g}_{j+1})$. Consider $(x_1,\cdots, x_j, t_{j+1}, \cdots, t_{n-1})$ a local chart in $\tilde{\Sigma}_j$ such that $(x_1,\cdots, x_j)$ is a local chart in $\Sigma_j$. It follows from Lemma \ref{l2} that 
\begin{eqnarray*}
\tilde{H}_j &=& \sum_{i,k=1}^{n-1}\hat{g}_{j+1}^{ik}(\tilde{B}_j)_{ik}\\
&=& \sum_{i,k=1}^{j}g^{ik}(B_j)_{ik}-\sum_{p=j+1}^{n-1}\frac{\nu_j(u_p)}{u_p}\\
&=& H_j-\sum_{p=j+1}^{n-1}\nu_j(\ln u_p)\\
&=& H_j-\nu_j(\ln \rho_{j+1})\\
&=& H_j-\langle \nabla_{j+1} \ln \rho_{j+1},\nu_j\rangle \\
\end{eqnarray*}
\noindent where $\nabla_{j+1}$ is the gradient in $(\Sigma_{j+1},g)$. We have that $\Sigma_j$ minimizes the  weight volume functional $V_{\rho_{j+1}}$, in particular, the $(\ln \rho_{j+1})$-mean curvature  of $\Sigma_j$ in $(\Sigma_{j+1},g)$ vanishes everywhere, this is, $H_j=\langle \nabla_{j+1} \ln \rho_{j+1},\nu_j\rangle .$ (See \cite{LX}). This implies that $\tilde{H}_j=0$. Therefore, $\tilde{\Sigma}_j$ is a free boundary minimal hypersurfaces in $(\hat{\Sigma}_{j+1},\hat{g}_{j+1})$.
\end{proof}

Denote by $S_j$ the second variation for weight volume functional $V_{\rho_{j+1}}$ in  $\Sigma_{j}$,  $\tilde{S}_j$ the second variation for volume functional of $\tilde{\Sigma}_j$ in $(\hat{\Sigma}_{j+1},\hat{g}_{j+1})$ and $\tilde{g}_j=\left.\hat{g}_{j+1}\right|_{\tilde{\Sigma}_j}$. 

\begin{Pro} For every $j=k,\cdots,n-1$, $\tilde{\Sigma}_j$ is a free boundary stable minimal hypersurfaces in $(\hat{\Sigma}_{j+1},\hat{g}_{j+1})$.
\end{Pro}
\begin{proof} Let $\varphi \in C^{\infty}(\Sigma_j)$. We have that
\begin{eqnarray*}
S_j(\varphi)&=&\int_{\Sigma_j}\left[|\nabla_j\varphi|^2-(|B_j|^2+Ric_{f_{j+1}}(\nu_j,\nu_j))\varphi^2\right]\rho_{j+1}dv_j\\
&-&\int_{\partial\Sigma_j}\varphi^2B^{\partial\Sigma_{j+1}}(\nu_j,\nu_j)\rho_{j+1}d\sigma _j
\end{eqnarray*}
\noindent where $Ric_{f_{j+1}}(\nu_j,\nu_j)=Ric_{j+1}(\nu_j,\nu_j)-(\nabla_{j+1}^2f_{j+1})(\nu_j,\nu_j)$, $f_{j+1}=\ln \rho_{j+1}$ (See \cite{LX}).  Here, $dv_j$ and $d\sigma_j$ are the volume forms of $(\Sigma_j,g)$ and $(\partial \Sigma_j,g)$, respectively.  Note that 
\begin{eqnarray*}
\nabla_{j+1}f_{j+1} &=&  \nabla_{j+1}\ln \rho_{j+1}\\
&=&\nabla_{j+1}\left(\sum_{p=j+1}^{n-1}\ln u_p\right)\\
&=&\sum_{p=j+1}^{n-1}\nabla_{j+1}\ln u_p\\
&=&\sum_{p=j+1}^{n-1}\frac{1}{u_p}\nabla_{j+1} u_p.
\end{eqnarray*}
It follows that
\begin{eqnarray*}
(\nabla_{j+1}^2f_{j+1})(\nu_j,\nu_j) &=& \left\langle \nabla_{\nu_j}\left(\nabla_{j+1}f_{j+1}\right),\nu_j\right\rangle\\
&=&  \left\langle \nabla_{\nu_j}\left( \sum_{p=j+1}^{n-1}\frac{1}{u_p}\nabla_{j+1} u_p \right),\nu_j\right\rangle\\
&=&  \sum_{p=j+1}^{n-1} \left\langle \frac{1}{u_p}\nabla_{\nu_j}\left(\nabla_{j+1} u_p\right) -\frac{\nu_j(u_p)}{u_p^2}\nabla_{j+1} u_p,\nu_j\right\rangle\\
&=& \sum_{p=j+1}^{n-1} \frac{1}{u_p} \left\langle \nabla_{\nu_j}\left(\nabla_{j+1} u_p\right),\nu_j\right\rangle- \sum_{p=j+1}^{n-1}\frac{1}{u_p^2}[\nu_j(u_p)]^2\\
&=& \sum_{p=j+1}^{n-1} \frac{1}{u_p} \left(\nabla_{j+1}^2u_p\right)(\nu_j,\nu_j)- \sum_{p=j+1}^{n-1}[\nu_j(\ln u_p)]^2\\
\end{eqnarray*}

From Lemmas \ref{l2} and \ref{l22} we obtain 
$$Ric_{f_{j+1}}(\nu_j,\nu_j)+|B_j|^2=Ric_{\hat{g}_{j+1}}(\nu_j,\nu_j)+|\tilde{B}_j|^2.$$

This implies that
\[S_j(\varphi)=\int_{\Sigma_j}\left(|\nabla_j\varphi|^2-Q_j\varphi^2\right)\rho_{j+1}dv_j-\int_{\partial\Sigma_j}\varphi^2B^{\partial\Sigma_{j+1}}(\nu_j,\nu_j)\rho_{j+1}d\sigma _j.\]
\noindent where
$$Q_j=Ric_{\hat{g}_{j+1}}(\nu_j,\nu_j)+|\tilde{B}_j|^2.$$

Consider now  $\Psi \in C^{\infty}(\tilde{\Sigma}_j)$. We have that
\begin{eqnarray*}
\tilde{S}_j(\Psi)&=&\int_{\tilde{\Sigma}_j}\left[|\nabla_{\tilde{g}_j}\Psi|^2-Q_j \Psi^2\right]dv_{\tilde{g}_j}-\int_{\partial\tilde{\Sigma}_j}\Psi^2\hat{B}_{j+1}(\nu_j,\nu_j)d\sigma _{\tilde{g}_j}\\
\end{eqnarray*}
\noindent where $dv_{\tilde{g}_j}$ and $d\sigma _{\tilde{g}_j}$ are the volume forms of $(\tilde{\Sigma_j},\tilde{g}_j)$ and $(\partial \tilde{\Sigma_j},\tilde{g}_j)$, respectively. From Lemma \ref{l21} we have that
\begin{eqnarray*}
\tilde{S}_j(\Psi) &=& \int_{\tilde{\Sigma}_j}\left(|\nabla_{\tilde{g}_j}\Psi|^2-Q_j \Psi^2\right)dv_{\tilde{g}_j} - \int_{\partial\tilde{\Sigma}_j}\Psi^2B^{\partial \Sigma_{j+1}}(\nu_j,\nu_j)d\sigma _{\tilde{g}_j}
\end{eqnarray*}

Furthermore, since $dv_{\tilde{g}_j}=\rho_{j+1}dv_j dt$ and $d\sigma _{\tilde{g}_j}=\rho_{j+1}d\sigma _j dt$, where $dt=dt_{j+2}\cdots dt_{n-1}$, we  have that
\begin{eqnarray*}
\tilde{S}_j(\Psi)&=&\int_{T^{n-j-1}}\left( \int_{{\Sigma}_j}\left(|\nabla_{\tilde{g}_j}\Psi|^2-Q_j\Psi^2\right)\rho_{j+1}dv_{j}\right)dt \\
&-&\int_{T^{n-j-1}}\left( \int_{\partial{\Sigma}_j}\Psi^2B^{\partial \Sigma_{j+1}}(\nu_j,\nu_j)\rho_{j+1}d\sigma _{j}\right)dt
\end{eqnarray*}

For each $\Psi \in C^{\infty}(\tilde{\Sigma}_j)$ define $F_{\Psi}: T^{n-j-1}\rightarrow \mathbb{R}$ by $F_{\Psi}(t)=S_j(\Psi_t)$, where for each $t\in T^{n-j-1}$ the function $\Psi_t\in C^{\infty}(\Sigma_j)$ is defined by $\Psi_t(x)=\Psi(x,t)$, $x\in\Sigma_j$. Note that
\begin{equation}\label{eq1}
\tilde{S}_j(\Psi) \geq \int_{T^{n-j-1}}F_{\Psi} dt.
\end{equation}

Since  $\Sigma_j$ minimizes the  weight volume functional  $V_{\rho_{j+1}}$ we have that $F_{\Psi}>0$ for every $\Psi \in C^{\infty}(\tilde{\Sigma}_j)$. It follows that $\tilde{S}_j(\Psi)>0$ for every $\Psi \in C^{\infty}(\tilde{\Sigma}_j)$. Hence, $\tilde{\Sigma}_j$ is a free-boundary stable minimal hypersurface in $(\hat{\Sigma}_{j+1},\hat{g}_{j+1})$.
\end{proof}

Note that the equality holds in (\ref{eq1}) if and only if $\Psi \in C^{\infty}({\Sigma}_j)$. So $S_j(\varphi)=\tilde{S}_j(\varphi),$ for every $\varphi\in C^{\infty}(\Sigma_j)$. It follows that
\begin{eqnarray*}
S_j(\varphi)&=&\int_{\Sigma_j}(|\nabla_j\varphi|^2-Q_j\varphi^2)\rho_{j+1}dv_j -\int_{\partial \Sigma_{j}}\varphi^2B^{\partial\Sigma_{j+1}}(\nu_j,\nu_j)\rho_{j+1}d\sigma_j\\
&=& -\int_{\Sigma_j}\varphi \tilde{L}_j(\varphi)\rho_{j+1}dv_j+ \int_{\partial\Sigma_j}\varphi\left(\frac{\partial \varphi}{\partial \eta_j}-\varphi B^{\partial\Sigma_{j+1}}(\nu_j,\nu_j)\right)\rho_{j+1}d\sigma_j
\end{eqnarray*}
\noindent for every $\varphi\in C^{\infty}(\Sigma_j)$, where $\tilde{L}_j:C^{\infty}(\Sigma_j)\rightarrow C^{\infty}(\Sigma_j)$ is a differential operator given by $\tilde{L}(\varphi)=\tilde{\Delta}_j\varphi+Q_j\varphi,$ where $\tilde{\Delta}_j$ denote the Laplacian operator of $(\tilde{\Sigma}_j,\hat{g}_{j+1})$.

Consider $\lambda_j$ the firt eingevalue of $S_j$ associeted the first eigenfunction $u_j$. We have that,
\begin{equation}\label{eq3}
\left\{
\begin{array}{ccccccc}
\tilde{L}_j(u_j)&=& -\lambda_ju_j \ \ \text{on} \ \ \Sigma_j\\

\displaystyle\frac{\partial u_j}{\partial \eta_j} &=& u_j B^{\partial\Sigma_{j+1}}(\nu_j,\nu_j)\ \ \text{on} \ \ \partial\Sigma_j\\

\end{array}
\right.
\end{equation}

\begin{lemma}\label{16} For every $j\leq p\leq n-1$ , we have that, in $\partial\Sigma_j$,
\[
B^{\partial \Sigma_{p+1}}(\nu_p,\nu_p)=\langle \nabla_j\log u_p, \eta_j\rangle.
\]
\end{lemma}
\begin{proof}It follows from (\ref{eq3}) that, in $\partial\Sigma_p$,
\[
B^{\partial\Sigma_{p+1}}(\nu_p,\nu_p)=\frac{1}{u_p}\frac{\partial u_p}{\partial \eta_p}=\langle \nabla_p\log u_p,\eta_p\rangle,
\]
\noindent for every $p=k,\cdots,n-1$. Consider $j\leq p\leq n-1$. Note that, in $\partial\Sigma_j$,
\[
B^{\partial\Sigma_{p+1}}(\nu_p,\nu_p)=\langle \nabla_p\log u_p,\eta_j\rangle,
\]
\noindent because we have $\eta_p=\eta_j$ in $\partial\Sigma_j$ (see remark \ref{obs1}). In $\Sigma_j$, we can write 
\[
\nabla_p\log u_p=\nabla_j\log u_p+\sum_{l=j}^{p-1}\langle \nabla_p\log u_p,\nu_l\rangle \nu_l.
\]
Hence, in $\partial\Sigma_j$, we have that
\[
B^{\partial\Sigma_{p+1}}(\nu_p,\nu_p)=\langle \nabla_j\log u_p,\eta_j\rangle +\sum_{l=j}^{p-1}\langle \nabla_p\log u_p,\nu_l\rangle\langle \nu_l,\eta_j\rangle.
\]
However, we have $\eta_j\perp \nu_l$ in $\partial\Sigma_j$, for every $j\leq l \leq n-1$. Therefore, \[
B^{\partial\Sigma_{p+1}}(\nu_p,\nu_p)=\langle \nabla_j\log u_p,\eta_j\rangle
\]
\end{proof}
\begin{lemma}\label{l4}  For $k\leq j\leq n-1$, the scalar curvature $\tilde{R}_j$ of $(\tilde{\Sigma}_j,\tilde{g}_{j})$ is given by
\[
\tilde{R}_j=R_j-2\sum_{p=j+1}^{n-1}u_p^{-1}\Delta_ju_p-2\sum_{j+1\leq p<q\leq n-1}\langle \nabla_j\log u_p,\nabla_j\log u_q\rangle.
\]

Equivalently,
\[
\tilde{R}_j=R_j-4\rho_{j+1}^{-\frac{1}{2}}\Delta_j(\rho_{j+1}^{\frac{1}{2}})-\sum_{p=j+1}^{n-1}|\nabla_j\log u_p|^2.
\]
\end{lemma}

\begin{lemma}\label{15} For $k\leq j\leq n-1$, the scalar curvature $\hat{R}_j$ of $(\hat{\Sigma}_{j},\hat{g}_{j})$ is given by
\begin{eqnarray*}
\hat{R}_j &=& R_j-2\sum_{p=j}^{n-1}u_p^{-1}\Delta_ju_p-2\sum_{j\leq p<q\leq n-1}\langle \nabla_j\log u_p,\nabla_j\log u_q\rangle\\
&=& \hat{R}_{j+1}+|\tilde{B}_j|^2+2\lambda_j\\
&=& R^M+\sum_{p=j}^{n+1}|\tilde{B}_p|^2+2\sum_{p=j}^{n+1}\lambda_p.
\end{eqnarray*}
\end{lemma}

\begin{proposition}\label{p1} If $R^M>0$ and $H^{\partial M}\geq 0$ then
\[
4\int_{\Sigma_j}|\nabla_j\varphi|^2dv_j> -2\int_{\partial\Sigma_j}\varphi^2 H^{\partial\Sigma_j}d\sigma_j
-\int_{\Sigma_j}\varphi^2R_jdv_j,
\]
for every $\varphi\in C^{\infty}(\Sigma_j)$ and $j=k,\cdots,n-1$.
\end{proposition}
\begin{proof}
Since $\Sigma_{j}$ minimizes the weighted volume functional $V_{\rho_{j+1}}$, we have that $S_j(\varphi)\geq 0,$ for every $\varphi\in C^{\infty}(\Sigma_j)$. It follows that,
\[ 
4\int_{\Sigma_j}|\nabla_j\varphi|^2\rho_{j+1}dv_j\geq
2\int_{\Sigma_j}c_j\varphi^2\rho_{j+1}dv_j+2\int_{\partial \Sigma_{j}}\varphi^2B^{\partial\Sigma_{j+1}}(\nu_j,\nu_j)\rho_{j+1}d\sigma_j,
\]
for every $\varphi\in C^{\infty}(\Sigma_j).$ From Gauss Equation we have that
\[Q_j=\frac{1}{2}(\hat{R}_{j+1}-\tilde{R}_j+|\tilde{B}_j|^2).\]

Since $R^M>0$, from lemma \ref{15}, we have that $\hat{R}_{i}>0$, for every $k\leq i\leq n-1$. It follows from the lemma \ref{l4} that

\[2Q_j>-R_j+4\rho_{j+1}^{-\frac{1}{2}}\Delta_j(\rho_{j+1}^{\frac{1}{2}})\]

Thus,
\begin{eqnarray*}
4\int_{\Sigma_j}|\nabla_j\varphi|^2\rho_{j+1}dv_j & > &-\int_{\Sigma_j}R_j\varphi^2\rho_{j+1}dv_j+ 4\int_{\Sigma_j}\rho_{j+1}^{\frac{1}{2}}\Delta_j(\rho_{j+1}^{\frac{1}{2}})\varphi^2dv_j\\
&  &+ 2\int_{\partial \Sigma_{j}}\varphi^2B^{\partial\Sigma_{j+1}}(\nu_j,\nu_j)\rho_{j+1}d\sigma_j,
\end{eqnarray*}
for every $\varphi\in C^{\infty}(\Sigma_j).$ Replacing $\varphi$ by $\varphi\rho_{j+1}^{-\frac{1}{2}}$ at the last inequality, we obtain that 
\begin{eqnarray*}
4\int_{\Sigma_j}|\nabla_j(\varphi\rho_{j+1}^{-\frac{1}{2}})|^2\rho_{j+1}dv_j &  > &-
\int_{\Sigma_j}R_j\varphi^2dv_j + 4\int_{\Sigma_j}\rho_{j+1}^{-\frac{1}{2}}\Delta_j(\rho_{j+1}^{\frac{1}{2}})\varphi^2dv_j \nonumber\\
& & +2\int_{\partial \Sigma_{j}}\varphi^2B^{\partial\Sigma_{j+1}}(\nu_j,\nu_j)d\sigma_j.
\end{eqnarray*}

Observe that
\[
\nabla_j(\varphi\rho_{j+1}^{-\frac{1}{2}})=\varphi\nabla_j\rho_{j+1}^{-\frac{1}{2}}+\rho_{j+1}^{-\frac{1}{2}}\nabla_j\varphi
\]

This implies que,
\[|\nabla_j(\varphi\rho_{j+1}^{-\frac{1}{2}})|^2=\rho_{j+1}^{-1}|\nabla_j\varphi|^2+\varphi^2|\nabla_j\rho_{j+1}^{-\frac{1}{2}}|^2+2\varphi\rho_{j+1}^{-\frac{1}{2}}\langle \nabla_j\rho_{j+1}^{-\frac{1}{2}},\nabla_j \varphi\rangle
\]

Thus,
\[
\rho_{j+1}|\nabla_j(\varphi\rho_{j+1}^{-\frac{1}{2}})|^2=|\nabla_j\varphi|^2+\varphi^2\rho_{j+1}|\nabla_j\rho_{j+1}^{-\frac{1}{2}}|^2+\langle \nabla_j\log\rho_{j+1}^{-\frac{1}{2}},\nabla_j (\varphi^2)\rangle
\]

Using integration by parts, we have that
\begin{eqnarray*}
\int_{\Sigma_j} \langle \nabla_j\log\rho_{j+1}^{-\frac{1}{2}},\nabla_j (\varphi^2)\rangle dv_j & = &-\int_{\Sigma_j}\varphi^2\Delta_j\log \rho_{j+1}^{-\frac{1}{2}}dv_j\\
&+& \int_{\partial \Sigma_j}\varphi^2\frac{\partial(\log\rho_{j+1}^{-\frac{1}{2}})}{\partial \eta_j}d\sigma_j\\
& = & +\int_{\Sigma_j}\varphi^2\rho_{j+1}^{-\frac{1}{2}}\Delta_j \rho_{j+1}^{\frac{1}{2}}dv_j\\
&-&\int_{\Sigma_j}\varphi^2|\nabla_j\log \rho_{j+1}^{\frac{1}{2}}|^2)dv_j\\
&-&\frac{1}{2}\int_{\partial \Sigma_j}\varphi^2\langle \nabla_j \log \rho_{j+1},\eta_j\rangle d\sigma_j\\
& = & -\int_{\Sigma_j}\varphi^2|\nabla_j\log \rho_{j+1}^{\frac{1}{2}}|^2dv_j\\
&+& \int_{\Sigma_j}\varphi^2\rho_{j+1}^{-\frac{1}{2}}\Delta_j \rho_{j+1}^{\frac{1}{2}}dv_j\\
&-&\frac{1}{2}\int_{\partial \Sigma_j}\varphi^2\langle \nabla_j \log \rho_{j+1},\eta_j\rangle d\sigma_j\\
\end{eqnarray*}

Then,
\begin{eqnarray*}
4\int_{\Sigma_j} \rho_{j+1}|\nabla_j(\varphi\rho_{j+1}^{-\frac{1}{2}})|^2dv_j & = & 4\int_{\Sigma_j}|\nabla_j\varphi|^2dv_j\\
&+& 4\int_{\Sigma_j}\varphi^2\rho_{j+1}|\nabla_j\rho_{j+1}^{-\frac{1}{2}}|^2dv_j\\
&-& 4\int_{\Sigma_j}\varphi^2|\nabla_j\log \rho_{j+1}^{\frac{1}{2}}|^2dv_j\\
& + &4\int_{\Sigma_j}\varphi^2\rho_{j+1}^{-\frac{1}{2}}\Delta_j \rho_{j+1}^{\frac{1}{2}}dv_j\\
&-& 2\int_{\partial \Sigma_j}\varphi^2\langle \nabla_j \log \rho_{j+1},\eta_j\rangle d\sigma_j\\
\end{eqnarray*}

Since,
\[ 
\nabla_j \rho_{j+1}^{-\frac{1}{2}}=-\rho_{j+1}^{-1}\nabla_j\rho_{j+1}^{\frac{1}{2}},
\]

\noindent we obtain that
\[
\rho_{j+1}|\nabla_j\rho_{j+1}^{-\frac{1}{2}}|^2=|\nabla_j\log \rho_{j+1}^{\frac{1}{2}}|^2.
\]

This implies that
\begin{eqnarray*}
4\int_{\Sigma_j} \rho_{j+1}|\nabla_j(\varphi\rho_{j+1}^{-\frac{1}{2}})|^2dv_j &=& 4\int_{\Sigma_j}|\nabla_j\varphi|^2dv_j\\
&+& 4\int_{\Sigma_j}\varphi^2\rho_{j+1}^{-\frac{1}{2}}\Delta_j \rho_{j+1}^{\frac{1}{2}}dv_j\\
&-& 2\int_{\partial \Sigma_j}\varphi^2\langle \nabla_j \log \rho_{j+1},\eta_j\rangle d\sigma_j\\
\end{eqnarray*}

Consequently, 
\begin{eqnarray*}
4\int_{\Sigma_j}|\nabla_j\varphi|^2dv_j & > &  2\int_{\partial \Sigma_{j}}\varphi^2\left(B^{\partial\Sigma_{j+1}}(\nu_j,\nu_j)+\langle \nabla_j \log\rho_{j+1},\eta_j\rangle\right)d\sigma_j\\
& - & \int_{\Sigma_j}R_j\varphi^2dv_j
\end{eqnarray*}

Since $H^{\partial M}_g\geq 0$, from the remark \ref{obs1} and lemma \ref{16} that 
\begin{eqnarray*}
4\int_{\Sigma_j}|\nabla_j\varphi|^2dv_j & > &  2\int_{\partial \Sigma_{j}}\varphi^2\left(\sum_{p=j}^{n-1}B^{\partial\Sigma_{p+1}}(\nu_p,\nu_p)\right)d\sigma_j -\int_{\Sigma_j}R_j\varphi^2dv_j\\
& = &  2\int_{\partial \Sigma_{j}}\varphi^2\left( H^{\partial M}_g-H^ {\partial \Sigma_j}\right)d\sigma_j - \int_{\Sigma_j}R_j\varphi^2dv_j\\
& \geq &  -2\int_{\partial \Sigma_{j}}\varphi^2H^ {\partial \Sigma_j} d\sigma_j-\int_{\Sigma_j}R_j\varphi^2dv_j
\end{eqnarray*}

Therefore,
\[
4\int_{\Sigma_j}|\nabla_j\varphi|^2dv_j> -2\int_{\partial\Sigma_j}\varphi^2 H^{\partial\Sigma_j}d\sigma_j-\int_{\Sigma_j}\varphi^2R_jdv_j,
\]
\noindent for every $\varphi\in C^{\infty}(\Sigma_j)$.
\end{proof}

\begin{theorem}\label{its}
Let $(M,\partial M,g)$ be a Riemannian $n$-manifold  such that $R^M>0$ and $H^{\partial M}\geq 0$. Consider the free boundary minimal $k$-slicing in $(M,g)$
\[\Sigma_k\subset \cdots \subset \Sigma_{n-1}\subset \Sigma_n=M.\]

Then:
\begin{enumerate}
\item[(1)] The manifold $\Sigma_j$ has a metric with positive scalar curvature and minimal boundary, for every $3\leq k\leq j\leq n-1$.
\item[(2)] If $k=2$, then the connected components of $\Sigma_2$ are disks.
\end{enumerate}
\end{theorem}
\begin{proof}
(1) Consider $j\in\{k,\cdots,n-1\}$, here $k\geq 3$. It follows from Proposition \ref{p1} that
\[-4k_j\int_{\Sigma_j}|\nabla_j\varphi|^2dv_j < 2k_j\int_{\partial\Sigma_j}\varphi^2 H^{\partial\Sigma_j}d\sigma_j+k_j\int_{\Sigma_j}\varphi^2 R_jdv_j,\]
\noindent for every $\varphi\in C^{\infty}(\Sigma_j)$ such that $\varphi\not\equiv 0$ and $k_j=\frac{j-2}{4(j-1)}>0.$ This implies that 
\[\displaystyle\int_{\Sigma_j}|\nabla_j\varphi|^2dv_j +2k_j\displaystyle\int_{\partial\Sigma_j}\varphi^2 H^{\partial\Sigma_j}d\sigma_j+k_j\displaystyle\int_{\Sigma_j}\varphi^2 R_jdv_j> (1-4k_j)\displaystyle\int_{\Sigma_j}|\nabla_j\varphi|^2dv_j,\]\noindent for every $\varphi\in H^1(\Sigma_j)$ such that $\varphi\not\equiv 0$. It follows that
\[\lambda_j=\inf_{0\not\equiv \varphi\in H^1(\Sigma_j)}\displaystyle\frac{\displaystyle\int_{\Sigma_j}|\nabla_j\varphi|^2dv_j +2k_j\displaystyle\int_{\partial\Sigma_j}\varphi^2 H^{\partial\Sigma_j}d\sigma_j+k_j\displaystyle\int_{\Sigma_j}\varphi^2 R_jdv_j}{\displaystyle\int_{\Sigma_j}\varphi^2dv_j}>0.\]

Therefore, there exists a metric in $\Sigma_j$ with positive scalar curvature and minimal boundary.

(2) From proposition \ref{p1} we have that
\[4\int_{\Sigma_2}|\nabla_2\varphi|^2dv_2> -2\int_{\partial\Sigma_2}\varphi^2 H^{\partial\Sigma_2}d\sigma_2-2\int_{\Sigma_2}\varphi^2 K dv_2,\] \noindent  for every $\varphi\in C^{\infty}(\Sigma_2)$ such that $\varphi\not\equiv 0$, because $R_2=2K_2$, where $K_2$ is the Gaussian curvature of $(\Sigma_2,g)$. In particular, for $\varphi\equiv 1$ we have that
\begin{equation}\label{e5}\int_{\partial\Sigma_2} H^{\partial\Sigma_2}d\sigma_2+\int_{\Sigma_2}K dv_2>0.\end{equation}

Let $S$ be a connected component of $\Sigma_2$. From inequality (\ref{e5}) and from Gauss-Bonnet theorem, we have that $\chi(S)>0.$ Therefore $S$ is a disk. 

\end{proof}

\section{Proof of inequality}
\begin{proposition}\label{pp11} There is a free boundary minimal $2$-slicing 

\[\Sigma_2\subset \Sigma_3\subset \cdots \subset \Sigma_{n+1}\subset (M,g),\]
such that $\Sigma_k$ is connected and the map $F_k:=\left.F\right|_{\Sigma_k}:(\Sigma_k,\partial \Sigma_k)\rightarrow (\mathbb{D}^2\times T^{k-2},\partial\mathbb{D}^2\times T^{k-2})$ has non-zero degree, for every $k=2,\cdots,n+1.$
\end{proposition}
\begin{proof}

Without loss of generality,we assume that $F$ is a smooth function.  Consider the projection $p_{j}:\mathbb{D}^2\times T^{j}\rightarrow S^1$ given by 
\[p_{j}(x,(t_1,\cdots,t_j))=t_{j},\]
\noindent for every $x\in \Sigma$ and $(t_1,\cdots,t_{j})\in T^{j}=\mathbb{S}^1\times\cdots \times \mathbb{S}^1$.

We will start constructing the manifold $\Sigma_{n+1}$. For this, define $f_{n}=p_{n}\circ F.$ It follows from the Sard's Theorem that there is $\theta_n\in S^1$ which is a regular value of $f_{n}$ and $\partial f_{n}$. Define
\[S_{n+1}:=f_{n}^{-1}(\theta_n)=F^{-1}(\mathbb{D}^2\times T^{n-1}\times\{\theta_n\}).\]

Note that $S_{n+1}\subset M$ is a properly embedded hypersurface which represents a non-trivial class in $H_{n+1}(M,\partial M)$ and 
\[\left.F\right|_{S_{n+1}}:(S_{n+1},\partial S_{n+1})\rightarrow (\mathbb{D}^2 \times T^{n-1}, \partial \mathbb{D}^2\times T^{n-1})\]
\noindent is a non-zero degree map. It follows from geometric measure theory that there is a properly embedded free-boundary smooth hypersuface $\Sigma_{n+1}'\subset M$ which minimizes volume in $(M,g)$ and represents the class $[S_{n+1}]\in H_{n+1}(M,\partial M)$. Since $\Sigma_{n+1}'$ and $S_{n+1}$ represent the same homology class in $H_{n+1}(M,\partial M)$, we have that 
\[\left.F\right|_{\Sigma_{n+1}'}:(\Sigma_{n+1}',\partial \Sigma_{n+1}')\rightarrow (\mathbb{D}^2 \times T^{n-1}, \partial \mathbb{D}^2\times T^{n-1})\]
\noindent has non-zero degree. Consider $\Sigma_{n+1}$ a connected component of $\Sigma_{n+1}'$ such that $F_{n+1}:=\left.F\right|_{\Sigma_{n+1}}:(\Sigma_{n+1},\partial \Sigma_{n+1})\rightarrow (\mathbb{D}^2 \times T^{n-1}, \partial \mathbb{D}^2\times T^{n-1})$ has non-zero degree. It follows from Lemma $33.4$ in \cite{simon} that $\Sigma_{n+1}$ is still a properly embedded free-boundary hypersurface which minimizes volume in $(M,g)$. Consider $u_{n+1}\in C^{\infty}(\Sigma_{n+1})$ a positive first eigenfunction for the second variation $S_{n+1}$ of the volume of $\Sigma_{n+1}$ in $(M,g)$. Define $\rho_{n+1}=u_{n+1}$. 

By a similar reasoning used to construct $\Sigma_{n+1}$, we obtain a properly embedded free boundary connected smooth hypersurface $\Sigma_{n}\subset \Sigma_{n+1}$ which minimizes the weighted  volume functional $V_{\rho_{n+1}}$ and  
\[F_{n}:=\left.F\right|_{\Sigma_{n}}:(\Sigma_{n},\partial \Sigma_{n})\rightarrow (\mathbb{D}^2 \times T^{n-2}, \partial \mathbb{D}^2\times T^{n-2})\]
\noindent has non-zero degree. Consider $u_{n}\in C^{\infty}(\Sigma_{n+1})$ a positive first eigenfunction for the second variation $S_{n}$ of $V_{\rho_{n+1}}$ on $\Sigma_{n}$. We then define $\rho_{n}=u_{n}\rho_{n+1}$ and we continue this process. 

\end{proof}

\begin{lemma}\label{lemmaa} We have that $\Sigma_2\in \mathcal{F}_M$.
\end{lemma}
\begin{proof}
From Theorem \ref{its} that $\Sigma_2$ is a disk. Since there is a non-zero degree map $F_2:(\Sigma_2,\partial \Sigma_2)\rightarrow (\mathbb{D}^2,\partial\mathbb{D}^2)$, we have that $\partial  \Sigma_2$ is a curve homotopically non-trivial in $\partial M$. Therefore,  $\Sigma_2\in \mathcal{F}_M$.
\end{proof}

\begin{lemma}\label{lemmaaa} We have that,

\[\frac{1}{2}\inf R^M |\Sigma_2|_g+\inf H^{\partial M}|\Sigma_2|_g \leq 2\pi.\]

Moreover, if equality holds then $R_2=\inf R^M$, $H^{\partial \Sigma_2}=\inf H^{\partial M}$ and $\left.u_k\right|_{\Sigma_2}$ are positive constants for every $k=2,\cdots, n+1$.
\end{lemma}
\begin{proof}
 From the remark \ref{obs1} and lemma \ref{16}
\[\inf H^{\partial M}\leq \sum_{p=2}^{n+1}\langle \nabla_2\log u_p,\eta_2\rangle + H^{\partial\Sigma_2}.\]

This implies that 
\begin{equation} \label{eqq1}
\inf H^{\partial M}|\partial \Sigma_2|_g\leq \sum_{p=2}^{n+1}\int_{\partial \Sigma_2}\langle \nabla_2\log u_pd\sigma_2,\eta_2\rangle + \int_{\partial \Sigma_2} H^{\partial\Sigma_2}d\sigma_2.
\end{equation}

From lemma \ref{15}, we have that
\begin{eqnarray*}
\hat{R}_2&=& R_2-2\sum_{p=2}^{n+1}u_p^{-1}\Delta_2u_p-2\sum_{2\leq p<q\leq n+1}\langle\nabla_2\log u_p,\nabla_2\log u_q\rangle\\
&=& R_2-2\sum_{p=2}^{n+1}u_p^{-1}\Delta_2u_p-\left|\sum_{p=2}^{n+1} X_p\right|^2+\sum_{p=2}^{n+1} |X_p|^2,
\end{eqnarray*}
\noindent where $X_p:=\nabla_2\log u_p$. Since
\[u_p^{-1}\Delta_2u_p=\Delta_2\log u_p+|X_p|^2,\]
\noindent we have that
\[\hat{R}_2=R_2-2\sum_{p=2}^{n+1}\Delta_2\log u_p-\left|\sum_{p=2}^{n+1} X_p\right|^2-\sum_{p=2}^{n+1} |X_p|^2.\]

Since $\hat{R}_2\geq \inf R^M$, we obtain
\begin{eqnarray*}
\frac{1}{2}\inf R^M|\Sigma_2|_g &\leq & \frac{1}{2}\int_{\Sigma_2}\hat{R}_2dv_2\\
&=& \frac{1}{2}\int_{\Sigma_2}R_2dv_2-\sum_{p=2}^{n+1}\int_{\Sigma_2}\Delta_2\log u_p dv_2\\
& & -\frac{1}{2}\int_{\Sigma_2}\left|\sum_{p=2}^{n+1} X_p\right|^2dv_2-\frac{1}{2}\sum_{p=2}^{n+1}\int_{\Sigma_2} |X_p|^2dv_2\\
& \leq &  \frac{1}{2}\int_{\Sigma_2}R_2dv_2-\sum_{p=2}^{n+1}\int_{\Sigma_2}\Delta_2\log u_p dv_2.
\end{eqnarray*}

It follows from Divergence Theorem that
\begin{equation}\label{eqq2}
\frac{1}{2}\inf R^M|\Sigma_2|_g \leq \frac{1}{2}\int_{\Sigma_2}R_2dv_2-\sum_{p=2}^{n+1}\int_{\partial\Sigma_2}\langle \nabla_2\log u_p,\eta_2\rangle d\sigma_2.
\end{equation}

By inequalities (\ref{eqq1}) and (\ref{eqq2}), we have that
\[\frac{1}{2}\inf R^M|\Sigma_2|_g +\inf H^{\partial M}|\partial\Sigma_2|_g  \leq \frac{1}{2}\int_{\Sigma_2}R_2dv_2+\int_{\partial \Sigma_2} H^{\partial\Sigma_2}d\sigma_2.\]

Therefore, from Gauss-Bonnet Theorem, we obtain

\[\frac{1}{2}\inf R^M|\Sigma_2|_g +\inf H^{\partial M}|\partial\Sigma_2|_g  \leq 2\pi\mathcal{X}(\Sigma_2)=2\pi.\]

However, note that if holds equality then the field $X_p=0$ for every $p=2,\cdots ,n+1.$ It follows that $\left.u_p\right|_{\Sigma_2}$ are positive constants for every $p=2,\cdots, n+1$. Consequently, $R_2=\hat{R}_2\geq \inf R^M$ and $H^{\partial\Sigma_2}\geq \inf H^{\partial M}$. Therefore, from Gauss-Bonnet theorem, we have that $R_2= \inf R^M$ and $H^{\partial\Sigma_2}=\inf H^{\partial M}$.

\end{proof}

\begin{corollary}
We have that,
\[\frac{1}{2}\inf R^M \mathcal{A}(M,g)+\inf H^{\partial M}\mathcal{L}(M,g)\leq 2\pi.\]

Moreover, if equality holds then $R_2=\inf R^M$, $H^{\partial \Sigma_2}=\inf H^{\partial M}$ and $\left.u_k\right|_{\Sigma_2}$ are positive constants for every $k=2,\cdots, n+1$.
\end{corollary}
\begin{proof}
We have that 
\[\frac{1}{2}\inf R^M \mathcal{A}(M,g)+\inf H^{\partial M}\mathcal{L}(M,g)\leq \frac{1}{2}\inf R^M |\Sigma|_g+\inf H^{\partial M}|\partial\Sigma|_g\]
\noindent for every $\Sigma\in\mathcal{F}_M$. From Proposition \ref{pp11} and Lemmas \ref{lemmaa} and \ref{lemmaaa} we have that there is $\Sigma_2\in\mathcal{F}_M$ such that 
\begin{equation}\label{eeqq1}\frac{1}{2}\inf R^M |\Sigma_2|_g+\inf H^{\partial M}|\Sigma_2|_g \leq 2\pi.\end{equation}
It follows that
\begin{equation}\label{eeqq2}\frac{1}{2}\inf R^M \mathcal{A}(M,g)+\inf H^{\partial M}\mathcal{L}(M,g)\leq 2\pi.\end{equation}
If the equality holds in (\ref{eeqq2}) then the equality holds in (\ref{eeqq1}). Therefore, from Lemma \ref{lemmaaa} we have that $R_2=\inf R^M$, $H^{\partial \Sigma_2}=\inf H^{\partial M}$ and $\left.u_k\right|_{\Sigma_2}$ are positive constants for every $k=2,\cdots, n+1$.
\end{proof}

\section{Proof of the Rigidity}
\begin{proof}
Without loss of generality, we can assume that $R_g\geq2$. Using an ideia in the Gromov-Lawsons paper on positive scalar curvature and mean-convex manifolds, we obtain that the doubling $DM$ of $M$ has a metric $g$ with $R_g\geq2$. Moreover,  if $F:(M,\partial M) \rightarrow (\mathbb{D}^2 \times T^n, \partial\mathbb{D}^2 \times T^n)$ is a non-zero degree map, then the induced map $DF: DM \rightarrow D\mathbb{D}^2 \times T^n$ has the same non-zero degree,
simply by looking at the preimage of a nonsingular point.  Hence, $DM$ admits a map to $\mathbb{S}^2\times T^n$ with non-zero degree, since $D\mathbb{D}^2=\mathbb{S}^2$. Now, we obtain that equality in (\ref{EQ1}) implies that the equality is achieved in the main inequality of the Theorem 1.1 in \cite{JZ} for our doubling manifold $DM$. Therefore, the rigidity part can be obtained from Theorem 1.1 in \cite{JZ}.

\end{proof}
\bibliographystyle{abbrv}
\bibliography{main}

\end{document}